\documentclass[12pt]{article}
\usepackage[margin=1in]{geometry}
\usepackage{enumerate}
\usepackage{amssymb}
\usepackage{array}
\usepackage{latexsym}
\usepackage{enumerate}
\usepackage{amsmath}
\usepackage{amsfonts}
\usepackage{amsthm}
\usepackage{graphicx}
\usepackage{psfrag}
\usepackage{comment}
\usepackage[english]{babel}
\usepackage[T1]{fontenc}
\usepackage{geometry}
\usepackage[disable]{todonotes}

\usepackage{hyperref}
\ifx\smallsetminus\undefined
\def\smallsetminus{\setminus}
\fi

\def\titlerunning#1{\gdef\titrun{#1}}
\makeatletter
\def\author#1{\gdef\autrun{\def\and{\unskip, }#1}\gdef\@author{#1}}

\makeatother
\def\email#1{\hspace*{4pt}{\em e-mail}: #1}

\theoremstyle{plain}
\newtheorem{conjecture}{Conjecture}

\newtheorem{prop}{Proposition}[section]
\newtheorem{theorem}[prop]{Theorem}

\newtheorem{lemma}[prop]{Lemma}
\newtheorem{example}[prop]{Example}
\theoremstyle{definition}
\newtheorem{remark}[prop]{Remark}
\newtheorem{definition}[prop]{Definition}

\newcommand\cP{\mathcal P}
\newcommand\cL{\mathcal L}

\newcommand\cF{\mathcal F}

\newcommand\cX{\mathcal X}
\newcommand\cH{\mathcal H}
\newcommand\cQ{\mathcal Q}

\newcommand{\fA}{\mathfrak{A}}

\newcommand{\PG}{\mathrm{PG}}
\newcommand{\AG}{\mathrm{AG}}

\newcommand{\cC}{\mathcal C}

\newcommand{\GF}{\mathrm{GF}}

\newcommand{\Tr}{\mathrm{Tr}}
\newcommand{\N}{\mathrm{N}}

\begin{document}
\titlerunning{}
\title{On regular  sets of affine type in finite Desarguesian planes  and  related codes}
\author{Angela Aguglia$^{\dagger}$ \footnote{Corresponding author.}
\and Bence Csajb\'ok\footnote{Dipartimento di Meccanica, Matematica e Management, Politecnico di Bari, Via Orabona 4, I-70125 Bari, Italy; {\it e-mail addresses:} angela.aguglia@poliba.it (A. Aguglia),\, bence.csajbok@poliba.it (B. Csajb\'ok).}
%\footnote{Dipartimento di Meccanica, Matematica e Management, Politecnico di Bari, Via Orabona 4, I-70125 Bari, Italy; \email{bence.csajbok@poliba.it}}
\and
Luca Giuzzi\footnote{DICATAM, University of Brescia, Via Branze 53,
  I-25123 Brescia, Italy; \email{luca.giuzzi@unibs.it}}}

\maketitle

\maketitle
\begin{abstract}

	%	Let $q=p^{2h}$ where $p$ is a prime.
	%	In this paper we determine the intersection sizes of  the Hermitian curve defined over
	%	$\GF(q^2)$ with suitable rational curves of degree $\sqrt{q}$.
	%	As a byproduct we  also exhibit a general constructions of an infinite family $\mathcal F$ of
	%	point sets in $\PG(2,q^2)$  with just $5$ intersection numbers,
	%        all congruent to $1$
	%	modulus $\sqrt{q}$, with respect to lines. These sets can be employed to get  $\sqrt{q}$-divisible, $[q^3+1,3,q^3-q-\sqrt{q}]_{q^2}$-codes with just $4$ non-zero weights for $q=4$ and $5$ non-zero weights otherwise.	We also determine their weight enumerator modulus $q^2$.
	%        Furthermore, each  point set  $X$ of $\mathcal F$ turns out to be regular, that is if $\cL$ is a parallel class of non-vertical lines, then the intersection numbers of $X$ w.r.t. lines of $\cL$ does not depend on the choice of $\cL$.
	%        Apart from $\mathcal F$ we provide some other two constructions of regular sets of affine type.
	%which we  do not expect  to be widespread.

\end{abstract}

%	Let $q=p^{2h}$ where $p$ is a prime.
In this paper, we consider point sets  of finite Desarguesian planes whose
multisets of intersection numbers with lines are the same for all but one exceptional parallel class of lines.
% the multisets of intersection numbers obtained from different parallel classes of lines is the same for all but one parallel class of lines.
We call such sets \emph{regular of affine type}. When the  lines of the exceptional parallel class have the same intersection numbers, then we call these sets \emph{regular of pointed type}. Classical examples are e.g. unitals; a detailed study and constructions of such sets with few intersection numbers is due to Hirschfeld and Sz\H onyi from 1991. We here provide some general construction methods for regular sets and describe a few infinite families.
The members of one of these families have the size as a unital and meet affine lines of $\PG(2,q^2)$ in one of $4$ possible intersection numbers, each of them congruent to $1$ modulus $\sqrt{q}$.
%We provide some general constructions of regular
%sets of affine type,  showing how regular sets of affine type in $\PG(2,q)$ can be used to constructed regular sets of affine type  in $\PG(2,q^h)$.
%In particular, by using trace and norm functions, we construct an infinite family of point sets in $\PG(2,q^2)$ with just 4 affine types that is 4 intersection numbers with respect to lines.

As a byproduct, we also  determine the intersection sizes of the Hermitian curve defined over $\GF(q^2)$, $q$ a square, with suitable rational curves of degree $\sqrt{q}$ and we obtain $\sqrt{q}$-divisible codes with $5$ non-zero weights. We also determine the weight enumerator of the codes arising from the
general constructions up to some $q$-powers.

%also exhibit a general constructions of an infinite family $\mathcal F$ of
%point sets in $\PG(2,q^2)$  with just $5$ intersection numbers,
%all congruent to $1$
%modulus $\sqrt{q}$, with respect to lines. These sets can be employed to get  $\sqrt{q}$-divisible, $[q^3+1,3,q^3-q-\sqrt{q}]_{q^2}$-codes with just $4$ non-zero weights for $q=4$ and $5$ non-zero weights otherwise.	 Furthermore, we  determine their weight enumerator modulus $q^2$.


\section{Introduction}
A classical problem in finite geometry is to construct and characterize sets which have some regularity in their intersection pattern with respect to lines; for example sets with few intersection numbers, sets with a unique intersection number modulus $p$ (the field characteristic), or sets such that for some parallel classes almost all lines of the same class meet the set in the same number of points (possibly modulus $p$), see \cite{renitent}. In this paper we consider point sets of finite Desarguesian planes such that the multisets of intersection numbers obtained from different parallel classes of lines are the same for all but one parallel class of lines. We call such sets \emph{regular of affine type}; their formal definition is given in Definition~\ref{d24}.

A point set $X$ of $\PG(2,q)$ is of type $(m_1,m_2,\ldots,m_k)$ if for each line $\ell$ of $\PG(2,q)$ there is  some $i\in \{1,2,\ldots,k\}$ such that  $|X\cap \ell|=m_i$. The numbers $m_1,m_2,\dots,m_k$ will be called the \emph{types} of $X$. It is hard,
in general, to find point sets with few types. In \cite{HSz} Hirschfeld and Sz\H{o}nyi introduced the notion of \emph{affine type} for those sets of $\PG(2,q)$ which admit at least one tangent line.
Assume that $P_0$ is a point of $X$ and $\ell_0$ is a tangent to $X$ at $P_0$, that is, a line of $\PG(2,q)$ such that $X \cap \ell_0 = \{P_0\}$. Up to a suitable choice of
reference, we may assume that $P_0$ is the common point $(\infty)$ of all \emph{vertical lines} of affine equation $x=\alpha$
of $\AG(2,q)$ and that $\ell_0=\ell_{\infty}$ is the line at infinity. Then $X$ is of affine type $(m_1,m_2,\ldots,m_k)$ if for each line $\ell \not\ni P_0$
we have $|X \cap \ell|=m_i$ for some $i\in \{1,2,\ldots,k\}$. The numbers $m_1,m_2,\ldots,m_k$ will be called the \emph{affine types} of $X$.
In~\cite{HSz} the authors proved various results about sets of affine type $(m,n)$ and showed infinite families of such point sets. If in addition
all of the vertical lines meet $X$ in the same number
of points, say $t+1$ with $t>1$,  then $X$ is  a \emph{set of pointed type} $[t;m,n]$ (note
that these are also sets of type $(1, t+1,m,n)$). The classical examples for such sets are the unitals of $\PG(2,q^2)$; they are exactly the sets of pointed type $[q;1,q+1]$. In \cite[Theorem 5.5 for $r=1$ and Theorem 5.8 for $\sqrt{q}\geq r > 1$]{HSz} Hirschfeld and Sz\H{o}nyi
constructed sets of pointed type $[rq;r(q-1),r(q-1)+q]$ in $\PG(2,q^2)$ for $q$ odd.
Still for $q$ odd,
the same authors constructed also some other sets of pointed type in $\PG(2,q)$ with a small number of affine types. These sets have $s+1$ affine types when  $s$ is an even divisor of $q-1$ and $2s+1$ affine types when $s>1$ is an odd divisor of $q-1$, \cite[Theorem 3.1]{HSz2}; see after Theorem~\ref{touching} for more details.
Another construction appears in~\cite{AK}, where sets of pointed type $[q;q-1,2q-1]$ are constructed in $\PG(2,q^2)$; see Remark~\ref{dafare}.

As usual, by $(d)$ we denote the common ideal point of the affine lines
$y=dx+b$ with slope $d\in \GF(q)$.
If $X$ is a set of affine type $(m,n)$ with distinguished point $P_0=(\infty)$ and with tangent $\ell_0=\ell_{\infty}$ then the number of $m$-secants and the number of $n$-secants incident with the direction $(d)\in \ell_{\infty}$ is the same for each $d\neq \infty$.
In this paper we will consider the following generalizations of this concept.

\begin{definition}
	\label{d24}
        A point set $X$ in $\PG(2,q)$ is \emph{regular of affine type $(m_1,m_2,\ldots,m_h)$} if there is a distinguished point $P_0$ in $X$ and a tangent $\ell_0$ of $X$ incident with $P_0$ such that:
	\begin{enumerate}
		\item[(i)] every line not through $P_0$ is an $m_i$-secant for some $i\in \{1,2,\ldots,h\}$;
\item[(ii)] the number of $m_i$-secants incident with $P$ is the same for each $P\in \ell_0 \setminus \{P_0\}$.
	\end{enumerate}
	The set $X$ is called \emph{regular of pointed type $[t;m_1,m_2,\ldots,m_h]$}
	for some $t>0$
	if in addition to (i) and (ii) it holds that
	\begin{enumerate}
		\item[(iii)] all the
			lines incident with $P_0$ other than $\ell_0$ are $(t+1)$-secants of $X$.
	\end{enumerate}
Finally, a set $X$ in $\PG(2,q)$  is said to be  \emph{of pointed type $[t;m_1,m_2,\ldots,m_h]$} if  properties (i) and (iii) hold.
\end{definition}

Clearly if $X$ is regular of pointed type then it is regular of affine type with the same parameters $(m_1,m_2,\dots,m_h)$. Assuming $P_0=(\infty)$ and $\ell_0=\ell_{\infty}$, trivial examples of regular sets of affine type are: subsets of a vertical line, the union of some vertical lines, a Baer subplane whose intersection with $\ell_{\infty}$ is $(\infty)$ and Korchm\'aros--Mazzocca arcs of type $(0,2,t)$ considered together with their nucleus (here the distinguished point is the nucleus), \cite{KM, GW}. The point sets constructed in~\cite[Theorem 3.1]{HSz2} are obtained from a pencil of touching conics. Although it is not mentioned there, sets obtained in this way are necessarily regular of pointed type, see Theorem~\ref{touching}. The touching conics idea can be applied to obtain some special examples involving internal and external points of a conic,
see Example~\ref{B}.

\medskip
Let $G$ denote a subgroup of the group of collineations $\mathrm{P}\Gamma\mathrm{L}(3,q)$ of $\PG(2,q)$ and denote by $O, O_1, \ldots, O_m$ some of the point orbits of $G$. Put $Y=\cup_{i=1}^k O_i$. Then the multiset $\{ |\ell \cap Y| : \ell \ni P\}$ is the same for each choice of $P\in O$. If $O=\ell_{\infty}$, then $Y$ is an affine set such that the number of $k$-secants of $Y$ incident with $P$ is the same for each $P\in \ell_{\infty}$. If $O=\ell_{\infty}\setminus (\infty)$ then $Y$ is regular of affine type, but not necessarily regular of pointed type, \cite{Korchmaros}.

 \todo{This is a new paragraph and I removed the paragraph that was here before.}

%One might ask why do we need a distinguished point $P_0$ in Definition \ref{d24} and why not search for affine sets $Y$ such that the number of $k$-secants of $Y$ incident with $P$ is the same for each $P\in \ell_{\infty}$. Let $\cal K$ be a maximal arc of $\AG(2,q)$, $P$ a point of $\cK$. Then $\cK$, $\cK\setminus \{P\}$, $\AG(2,q)\setminus \cK$, $(\AG(2,q)\setminus \cK)\cup \{P\}$  are such sets. Of course, if we add $(\infty)$ to such a set then we obtain regular sets of affine type $(0,n)$, $(0,n-1,n)$, $(q,q-n)$, $(q,q-n+1,q-n)$, respectively. Apart from some trivial cases (the whole affine plane, one point, the whole affine plane minus one point) we do not know whether there are examples when $q$ is odd, or examples not related to maximal arcs when $q$ is even. For $q$ odd, the nonexistence of nontrivial examples would be a natural generalization of the main result from \cite{maxarc}.

\medskip

In Section~\ref{psets} we consider some general constructions of regular
sets of affine type. In particular,
in Theorem~\ref{largerplane} and Theorem~\ref{largerplane2} we show how regular sets of affine type $(m_1,m_2,\ldots,m_k)$ in $\PG(2,q)$ may be used to constructed regular sets of affine type $(m_1,m_2,\ldots,m_k)$ in $\PG(2,q^h)$. This method is a mixture of ideas from~\cite{GW} and~\cite{Caserta}.

\medskip

By $\Tr$ and $\N$ we will denote the $\GF(q^2) \rightarrow \GF(q)$ functions $x\mapsto x+x^q$ and $x \mapsto x^{q+1}$, respectively. For some additive function $f \colon \GF(q^2) \rightarrow \GF(q^2)$ consider the algebraic plane curve of affine equation
\begin{equation}
	\label{primo}
	\Tr(y+f(x))=\N(x).
\end{equation}
In Theorem \ref{trnorm} we show that the set $X$ of the projective points of this curve is regular of pointed type in $\PG(2,q^2)$. Note that this result does not say anything about the affine types, or about the number of affine types. For certain choices of $f$ the resulting point set is a unital and, according to a non-exhaustive computer search for small values of $q$, when $X$ is not a unital then we have at least $4$ affine types (except when $q$ is even and $f(x)=ax^2$, see Remark \ref{dafare} where the non-additive choice $f(x)=ax^2$ is discussed for $q$ odd).  Up to equivalence, we found a unique infinite family with $4$ affine types,  obtained with the choice $f(x)=ax^{\sqrt{q}}$ whenever $q$ is a square prime power and $a\in \GF(q^2)^*$. This case is particular not only because there are few affine types but also because they are all congruent to $1$ modulus $\sqrt{q}$ and the point set $X \cup \{(\infty)\}$ meets each line of the plane in $1$ modulus $\sqrt{q}$ points. This $1$ modulus $p$ property holds only for certain choices of the additive function $f$.

In Section~\ref{hcurve} we determine the affine types of a special point-set of the aforementioned type.
This is the most laborious part of our work and it is done separately for the $q$ odd and $q$ even cases.
The main result of this section is the following.
\begin{theorem}
	\label{t:32}
	Let $q$ be a  square prime power and $a\in \GF(q^2)^*$. Let $\Gamma_{a}$  denote the algebraic plane curve of affine equation
	\begin{equation}
		\label{sss}
		\Tr(y+ax^{\sqrt{q}})=\N(x).
	\end{equation}
	%	for $q$ odd and of affine equation
	%	\begin{equation}\label{sssp}
	%		\Tr(y+ax^{\sqrt{q}})=2\mathrm{N}(x),
	%	\end{equation}
	%	for $q$ even, where  $\Tr$ and $\mathrm{N}$ denote respectively
	%	the trace and the norm of $\GF(q^2)$ over $\GF(q)$.
	%	
	Then the set of projective points of $\Gamma_a$ in $\PG(2,q^2)$ is a regular $(q^3+1)$-set of pointed type \[ [q;\ q-2\sqrt{q}+1, \ q-\sqrt{q}+1, q+1, \ q+\sqrt{q}+1] .\]
\end{theorem}
Using Theorem~\ref{t:32}, we  are able to describe the intersection between an Hermitian curve and a special family of curves of degree $\sqrt{q}$.

\begin{theorem}
	\label{main-1}
	Let $q$ be a  square prime power and let $a,m,d\in\GF(q^2)$, $a\neq 0$. Denote by
	$\cC(a,m,d)$ the curve of affine equation $y=ax^{\sqrt{q}}+mx+d$.
	Then the curves $\cC(a,m,d)$ meet the
	Hermitian curve $y^q+y=x^{q+1}$ of $\PG(2,q^2)$ in the following
        number of points:
        %point sets of the following size:
	\[ q-2\sqrt{q}+1, \ q-\sqrt{q}+1, \ q+1, \ q+\sqrt{q}+1.\]
\end{theorem}
We refer the reader to~\cite{Beelen2} for information about
the number of $\GF(q^2)$-rational points in the intersection of a non-degenerate Hermitian surface with a surface of degree-$d$, where Sørensen’s conjecture about the maximal number of such points when $d\leq q$ is proven.

In Section~\ref{codes} we apply Theorem~\ref{t:32} to study the
projective linear codes associated to $\Gamma_{a}$.
These codes are $\sqrt{q}$-divisible with only $5$ non-zero weights (when $q=4$ then with $2$ non-zero weights if $\Gamma_a$ is a unital and with $4$ non-zero weights otherwise). We also discuss the corresponding weight enumerators, which are important  tools in coding theory, since they contain some crucial information to estimate the
actual error-correcting capability and the probability of error-detection and correction of the code with respect to some channels.

%
%By considering a suitable birational transformation, it is possible
%to provide an alternative description of these configurations as
%the intersection pattern of suitable norm-trace curves with the
%lines of the projective plane.

%The configurations arising from corollary~\ref{t:32} can be
%described by a generalization of the notion of pointed-sets,
%originally introduced by Hirschfeld and Sz\H{o}nyi in~\cite{HSz}.

%The structure of the paper is as follows.
%In Section~\ref{psets} we introduce the definition of a regular set of affine type and  we give some  example as well as we provide an  example of a regular set of pointed type.
%In Section~\ref{hcurve} we
%prove Theorem~\ref{main-1} and in Section~\ref{pc32}, by using Theorem~\ref{main-1}
%we prove Theorem~\ref{t:32}  providing a family $\mathcal F$ of regular sets of pointed type with  few affine types.
%Finally, in Section~\ref{codes} we consider codes arising
%from $\mathcal F$. These turn out to have good parameters; in particular they have just $4$ non-zero weights for $q=4$ and $5$ non-zero weights otherwise and, in both cases all of them are divisible codes by $\sqrt{q}$. We also discuss the corresponding weight enumerators.
%We recall that the weight enumerator is an important research tool in coding theory, since it contains some crucial information as to estimate the error-correcting capability and the probability of error-detection and correction of the code with respect to some algorithms.

\section{General constructions of regular sets of affine type}
\label{psets}
First we show some general results on how to construct  new regular sets of affine type starting from a given regular set (of affine type) $X$. We will always assume that the distinguished point is $(\infty)$ and that $\ell_{\infty}$ is a tangent to $X$.

\begin{prop}
	\noindent
	\begin{enumerate}
        \item
          If $X$ is (regular) of affine type $(m_1,m_2,\ldots,m_h)$ in $\PG(2,q)$, then
          $(\AG(2,q)\setminus X) \cup \{(\infty)\}$ is (regular) of affine type $(q-m_1,q-m_2,\ldots,q-m_h)$.
		\item If $X$ is (regular) of pointed type $[t;m_1,m_2,\ldots,m_h]$ in $\PG(2,q)$, then $(\AG(2,q)\setminus X) \cup \{(\infty)\}$ is (regular) of pointed type $[q-t;q-m_1,q-m_2,\ldots,q-m_h]$.
	\end{enumerate}
	Clearly the same arguments hold in non-Desarguesian finite planes as well.
\end{prop}
\qed\par
The next construction is motivated by~\cite{GW} and~\cite{Caserta}. For $s=h-1$, this is the same as the construction of \cite[Section 3]{GW}. When $S$ is a $\GF(p)$-subspace of $\GF(q)\times \GF(q)$, then it is essentially the construction in~\cite[Section 4]{Caserta}. Our proof here works without assuming the additivity of $S$.

\begin{theorem}
	\label{largerplane}\todo{Statement reformulated.}
	Let $S=\{(x_k,y_k)\}_k \subseteq \AG(2,q)$ be a point set such that $S \cup \{(\infty)\}$ is of affine type $(m_1,m_2,\ldots,m_g)$.
	Let $v_1,\ldots,v_r$ denote integers such that each vertical line is incident with $v_i$ points of $S$ for some $i\in \{1,2,\ldots,r\}$. Take a non-trivial $s$-dimensional $\GF(q)$-subspace $I$ of $\GF(q^h)$, with $I\cap \GF(q)=\{0\}$.
	Consider the point set
	\[S'=\{(x_k,y_k+i) : i \in I,\, (x_k,y_k)\in S\} \subseteq \AG(2,q^h)\]
	of size $q^s|S|$. Then the vertical lines of $\AG(2,q^h)$ meet $S'$ in either $0$ or in $q^s v_i$ points, for  $i\in \{ 1,2,\ldots,r \}$.
	Non-vertical lines of $\AG(2,q^h)$ meet $S'$ in either $0$, $1$ or in $m_i$ points, $i\in \{ 1,2,\ldots,g \}$.
	In particular $S' \cup \{(\infty)\}$ is of affine type $(0,1,m_1,m_2,\ldots,m_g)$ and in $\ell_{\infty} \setminus (\infty)$ there are
	\begin{enumerate}[(1)]
        \item $q^{s+1}$ directions $(d)$ incident with $q^h-q^{s+1}$ affine lines that do not meet $S'$ and with $A_{i,d}\,q^s$ affine lines meeting $S'$ in $m_i$ points. These are exactly the directions $(d)$ with $d\in \GF(q)\oplus I$ and for each such $d$ put $d=d_0+d_1$ with $d_0\in \GF(q)$ and $d_1 \in I$. Then for $i\in\{1,2,\ldots,g\}$,  $A_{i,d}$ is the number of affine $m_i$-secants of $S$ in $\PG(2,q)$ incident with the direction $(d_0)\neq(\infty)$ and, $\sum_{i=1}^gA_{i,d}=q$;
		\item $q^h-q^{s+1}$ directions incident with $q^s|S|$ tangents to $S'$ and with  $q^h-q^s|S|$ affine lines that do not meet $S'$.
	\end{enumerate}
\end{theorem}
\begin{proof}
  The statement on the vertical lines is trivial, it remains to determine the size of $S' \cap \ell$, where $\ell$ is the line of equation $y=dx+b$.
  This is the same as the cardinality of the set
	\[\{ (k,i) \in \{1,2,\ldots,|S|\}\times I : y_k+i = d x_k + b\}.\]
	Let $I'$ be an $(h-1)$-dimensional $\GF(q)$-subspace of $\GF(q^h)$ such that $I \subseteq I'$ and
	$\GF(q)\oplus I'=\GF(q^h)$.
	For any $e\in \GF(q^h)$ put $e=e_0+e_1$ with $e_0 \in \GF(q)$ and $e_1 \in I'$.

	The condition $y_k+i = d x_k + b$ can be written as
	$y_k+i=(d_0+d_1)x_k+b_0+b_1$, that is,
	\[y_k-d_0 x_k-b_0 = d_1x_k+b_1-i.\]
	Here the left hand side is in $\GF(q)$ while the right hand side is in $I'$, thus equality holds if and only if
	\begin{equation}
		\label{0eq1}
		y_k-d_0x_k-b_0=0
	\end{equation}
	and
	\begin{equation}
		\label{0eq2}
		d_1 x_k+b_1-i=0.
	\end{equation}
	Let $\ell_0$ denote the line of equation $y=d_0x+b_0$ in $\AG(2,q)$ and let
	$w\in \{1,\ldots,g\}$ be such that
	\[m_w=|S \cap \ell_0|=|\{ (x_k,y_k) : y_k-d_0x_k-b_0=0\}|.\]

	\begin{enumerate}
		\item[(i)] If $d_1\in I$ and $b_1 \in I' \setminus I$ then \eqref{0eq2} does not have a solution and hence $\ell$ meets $S'$ in $0$ points.
		\item[(ii)]
			If $d_1 \in I$ and $b_1 \in I$ then for every solution $(x_k,y_k)$ of \eqref{0eq1} there corresponds a unique solution of \eqref{0eq2} and hence $\ell$ meets $S'$ in $m_w$ points.
		\item[(iii)] If $d_1 \in I' \setminus I$ then $\ell$ meets $X'$ in at most one point. Indeed, if we had $d_1 x_k+b_1=i$ and $d_1 x_{k'} + b_1 = i'$ for some $i,i' \in I$ and $k,k' \in \{1,2,\ldots,|S|\}$, then $d_1(x_k-x_{k'})=i-i'$ and hence $d_1 \in I$ would follow, a contradiction.
	\end{enumerate}

Since there are $q(q^{h-1}-q^s)$ pairs $(d_0,d_1)$ such that $d_1\in I' \setminus I$, part {\it (2)} of the theorem follows from (iii).
Since there are $q(q^{h-1}-q^s)$ pairs $(b_0,b_1)$ such that $b_1 \in I' \setminus I$, it follows that each of the $q^{s+1}$ directions $(d)$ such that $d=d_0+d_1$, $d_1 \in I$ is incident with $q^h-q^{s+1}$ affine lines that do not meet $S'$ (cf. (i)). On the other hand, if $d_1 \in I$ and $b_1 \in I$, then lines with equation $y=dx+b$ meet $S'$ in the same number of points as the line with equation $y=d_0 x+b_0$ meet $S$ (cf. $(ii)$). As there are $q^s$ possible values for $b$ (or choices for $b_1$) if we fix $b_0$,  this proves part {\it (1)}.
\todo{New paragraph.}

\end{proof}

\begin{theorem}
	\label{largerplane2}
	If in Theorem \ref{largerplane} the set $S\cup\{(\infty)\}$ is regular of affine type $(m_1,\ldots,m_g)$ in $\PG(2,q)$ and $s=h-1$ then $S' \cup \{(\infty)\}$ is regular of affine type $(m_1,\ldots,m_g)$ in $\PG(2,q^h)$. Moreover, the number of non-vertical, affine $k$-secants of $S'$ is a multiple of $q^{2h-1}$ for every integer $k$. \qed
\end{theorem}

\begin{example} Theorem~\ref{largerplane} can also be applied to construct sets with few types. Indeed, if $S$ of $\PG(2,q)$ is of type $(0,1,n)$, then $S'$ of $\PG(2,q^h)$ is of type $(0,1,n,n q^s)$.
\end{example}

%\noindent
%{\bf Questions.} Do you know any other examples? Can the methods of \cite{HSz} yield new examples with more than two affine types? Can our method, applied not to the Hermitian curve but to other point sets yield new examples? (Note that the hard part was to prove that the affine types are $q-2\sqrt{q}+1$, $q-\sqrt{q}+1$, $q+1$, $q+\sqrt{q}+1$. It was easy to see that the point set is regular of pointed type.)

We are going to provide two further  constructions of regular sets of pointed type in $\PG(2,q)$. The next construction is basically the same as in \cite{HSz,HSz2}. The difference is that in those papers $B$ is taken as a special subset of $\GF(q)$ and the regularity is not explicitly
indicated there.

\begin{theorem}
	\label{touching}
	For $b\in \GF(q)$, $q$ odd,
	let $P_b$ denote the conic of equation $yz=x^2+bz^2$ in $\PG(2,q)$.
	For $B\subseteq\GF(q)$ consider
	\[X(B):=\cup_{b\in B}\, P_b.\]
	Then $X(B)$ is regular of pointed type.
\end{theorem}
\begin{proof}
	If $B=\{b_1,b_2,\ldots,b_h\}$ then the vertical lines meet $X(B)$ in
	$h+1$ points. The line $\ell \colon y=mx+d$ meets $X(B)$ in $2\alpha_{m,d} + \beta_{m,d}$
	points where
	\[\alpha_{m,d}=|\{ b \in B : m^2-4(b-d) \text{ is a non-zero square in } \GF(q)  \}|,\]
	\[\beta_{m,d}=|\{b \in B : m^2-4(b-d)=0\}|.\]
	Since in the multiset $\{m^2+4r : r \in \GF(q)\}$ each element of $\GF(q)$ is represented exactly once, it follows that
	the multiset $\{2\alpha_{m,d}+\beta_{m,d} : d \in \GF(q)\}$ does not depend on the choice of $m$ and hence $X(B)$ is regular of pointed type.
\end{proof}

In~\cite[Theorem 3.1]{HSz2} it is shown that if
$B=\{vu^s : u \in \GF(q)\}$ for some fixed non-square $-v\in \GF(q)$ and $s \mid q-1$, $s$ even, then $B$ is of pointed type $[(q-1)/s+1; 1,m_1,\ldots,m_s]$. Note that with $s=2$ we obtain the same parameters as in the third subcase of Example!\ref{B} (cf.\ the first comment of page 503 of \cite{HSz2} and \cite[Section 3.4]{Barlotti} from Barlotti, where these sets are considered as complete $(k,n)$-arcs).
On the other hand, if $s$ is an odd divisor of $q-1$, then $X(B)$ is of pointed type $[(q-1)/s+1;(q-1)/s+1,m_1,\ldots,m_{2s}]$.

\begin{example}
	\label{B}
	In $\PG(2,q)$ with $q$ odd:
	\begin{enumerate}
		\item the set of the interior points of an oval together with  one point of the oval is regular of pointed type $[\frac12(q-1);0,\frac12 (q-1),\frac12 (q+1)]$;
		\item the set of exterior points of an oval together with one point of the oval is regular of pointed type $[\frac12(q-1);q-1,\frac12 (q-3),\frac12 (q-1)]$;
		\item the set of interior points of an oval united with
                  the set of all the points of the oval is regular of pointed type $[\frac12(q+1);1,\frac12 (q+3),\frac12 (q+1)]$;
		\item  the set of exterior points of an oval united with the set of all the points of the oval is regular of pointed type $[\frac12(q+1);q,\frac12 (q+1),\frac12 (q-1)]$.
	\end{enumerate}
\end{example}

Finally, we consider the pointed sets arising from~\eqref{primo} mentioned in the Introduction.

\begin{theorem}
	\label{trnorm}
	If $f$ is an additive $\GF(q^2) \rightarrow \GF(q^2)$ function then the set of projective points of the algebraic plane curve $X$ of affine equation
	\[\Tr(y+f(x))=\N(x)\]
	is a regular set of pointed type in $\PG(2,q^2)$. Moreover, in every parallel class of lines the number of $k$-secants to $X$ is a multiple of $q$ for each integer $k$.
\end{theorem}
\begin{proof}
	It is immediate to see that vertical lines meet $X$ in $q$ affine points.
	After substituting $y=mx+d$ to determine the number of common points of $X$ and the line $\ell \colon y=mx+d$ we obtain
	\[m^qx^q+d^q+mx+d+\Tr(f(x))=x^{q+1}.\]
	The number of solutions of this equation remains the same after replacing $x$ by $x+m^q$, thus we obtain:
        \begin{multline}
	m^q(x^q+m)+d^q+m(x+m^q)+d+\Tr(f(x)+f(m^q))=(x^q+m)(x+m^q) \Leftrightarrow \\
m^qx^q+m^{q+1}+d^q+mx+m^{q+1}+d+\Tr(f(x)+f(m^q))=x^{q+1}+x^qm^q+mx+m^{q+1} \Leftrightarrow \\
	\label{disp}
		m^{q+1}+\Tr(f(m^q))+d+d^q=x^{q+1}-\Tr(f(x)).
	\end{multline}
Denote by $t_{m,d}$ the number of different solutions in $x$ of the  equation \eqref{disp}.
	Recall that  $d\mapsto d+d^q$ is $q$-to-$1$, hence in the multiset $\{ m^{q+1}+\Tr(f(m^q))+d+d^q : d \in \GF(q^2)\}$ each element of $\GF(q)$ is represented exactly $q$ times. It
        it follows that
	the multiset $\{ t_{m,d} : d \in \GF(q^2)\}$ does not depend on the choice of $m$ and each of its elements is represented  a multiple of $q$ times. This completes the proof.
\end{proof}

\begin{remark}
	\label{dafare}
	The points of $\Tr(y+f(x))=\N(x)$ give unitals or regular sets of pointed type $[q;q-1,2q-1]$ when $f(x)=ax^2$, $a\in \GF(q)^*$, $4\N(a)\neq 1$, $q$ odd.
	Indeed, in this case the map $x \mapsto x^q-2ax$ is a bijection of $\GF(q^2)$. The number of solutions of $\Tr(y+ax^2)=\N(x)$, with $y=mx+b$, is the same as the number of solutions of $\Tr(m(x+\alpha)+b+a(x+\alpha)^2)=\N(x+\alpha)$ where $\alpha$ is chosen such that $\alpha^q-2a\alpha=m$. This equation reads as
	\[\Tr(m\alpha+a\alpha^2)-\N(\alpha)+\Tr(b)=\N(x)-\Tr(ax^2).\]
	If we view $\Tr(m\alpha+a\alpha^2)-\N(\alpha)$ as a constant and recall that $\Tr$ is $q$-to-$1$, it is immediate to see that the multiset of the number of solutions in $x$ (as $b$ runs through $\GF(q^2)$) does not depend on the choice of $m$.

	In~\cite{AK} a very similar construction, related to \emph{Ebert’s discriminant condition} (\cite{unital1, unital2}), was considered. Applying ideas from there it can be seen that the number of solutions of $\Tr(mx+b+ax^2)=\N(x)$ corresponds to the number of affine points of the conic of equation
	\[x_0^2(2a_0-1)+x_1^2\varepsilon^2(2a_0+1)+4x_0x_1\varepsilon^2 a_1+2x_0m_0+2\varepsilon^2m_1+2b_0=0\]
	of $\PG(2,q)$, where $\varepsilon^{q-1}=-1$ and for each $z\in \GF(q^2)$ we write $z=z_0+\varepsilon z_1$ with $(z_1,z_2)\in\GF(q)$. When $4\N(a)=1$ then this conic has a unique point at infinity and hence it has $0$, $q$, or $2q$ affine points, thus we obtain a set of pointed type $[q;0,q,2q]$ (which is not regular). On the other hand if $4\N(a)\neq 1$, then to determine the intersection numbers it is enough to consider the case $m=0$. The conic is reducible if and only if $\Tr(b)=0$ and it has $1$ affine point when $4\N(a)-1$ is a square and $2q-1$ affine points when $4\N(a)-1$ is a non-square in $\GF(q)$. When the curve is irreducible then it has $q+1$ affine points when $4\N(a)-1$ is a square and $q-1$ affine points when $4\N(a)-1$ is a non-square. Similar ideas work also when $q$ is even and $f(x)=ax^2$.
	In Section~\ref{hcurve} we shall examine the case $q$ square, $f(x)=x^{\sqrt{q}}$.
\end{remark}

\section{Proofs of Theorems \ref{t:32} and \ref{main-1}}
\label{hcurve}

\subsection{Proof of Theorem~\ref{t:32}}

\paragraph{The odd $q$ case.}
The affinity $\pi \colon (x,y) \mapsto (\alpha x, y)$  maps the points of the curve $\Tr(y+f(\alpha x))=\N(\alpha)\N(x)$ to the points of the curve $\Tr(y+f(x))=\N(x)$.
Take some $\alpha$ such that $\N(\alpha)=2$, with $f(x)=ax^{\sqrt{q}}$.
It follows that the set of points of the curve $\Tr(y+a\alpha^{\sqrt{q}} x^{\sqrt{q}})=2\N(x)$ is equivalent to $\Gamma_a$, which is regular of pointed type by Theorem~\ref{trnorm}.
Since $\{a\alpha^{\sqrt{q}} : a\in \GF(q^2)^*\}=\GF(q^2)^*$, to prove Theorem \ref{t:32} it is enough to determine the size of the  intersections
with lines of $\AG(2,q^2)$
of the curve $\Lambda_a$ of equation
\[\Tr(y+ax^{\sqrt{q}})=2\N(x)\]
with  $a\in\GF(q)^*$.

The constant $2$ at the right hand side is harmless and just to simplify the computations.
Since $\pi$ fixes $(\infty)$, the vertical lines meet $\Lambda_a$ in the same number of points as they met $\Gamma_a$, that is, in $q$  affine points. Since $\Lambda_a$ is regular of pointed type, it is enough to calculate the size of intersections of $\Lambda_a$ with horizontal lines. So, after substituting $y=d$, we need to determine the number of solutions of the following equations as $d$ varies in $\GF(q^2)$:
\begin{equation}
	\label{eq21}
	a^qx^{\sqrt{q}q}+ax^{\sqrt{q}}+d^q+d-2x^{q+1}=0.
	%a^qx^{2q}+m^qx^q+b^qx^{q+1}+d^q=bx^{q+1}+ax^2+mx+d.
\end{equation}

We can replace $x$ with $x a^{\sqrt{q}}$ in~\eqref{eq21} without changing the number of solutions. So we end up with the following
\[ \N(a)^{1-\sqrt{q}}x^{\sqrt{q}q}+\N(a)^{1-\sqrt{q}} x^{\sqrt{q}}+d^q/N(a)^{\sqrt{q}}+d/\N(a)^{\sqrt{q}}-2x^{q+1}=0.\]
In other words, the number of points of $\Lambda_a$ in common with the line $y=d$ is the same as the number of points of $\Lambda_{\N(a)^{1-\sqrt{q}}}$
in common with the line $y=d/\N(a)^{\sqrt{q}}$.
This means that in~\eqref{eq21} we may assume without loss of generality
$a \in \GF(q)^*$.
%Actually, even if we will not use this, we might also assume $a^{\sqrt{q}+1}=1$.
For any $a\in \GF(q)^*$, denote by $M_{a,d}$ the number of solutions of~\eqref{eq21}.

%\underline{First assume that $q$ is odd.}
Fix now a primitive element $\beta$ of $\GF(q^2)$
and put $\varepsilon=\beta^{(q+1)/2}$. Then,
$\varepsilon^q=-\varepsilon$ and $\varepsilon^2$ is a primitive
element of $\GF(q)$; also $\{1,\varepsilon\}$ is a basis of
$\GF(q^2)$ regarded as a vector space over $\GF(q)$.
The elements of $\GF(q^2)$ shall always be written as  linear
combinations with respect to this basis, that is,
$z=\hat{z}_0+\hat{z}_1\varepsilon$, with $z\in \GF(q^2)$ and
$\hat{z}_0,\hat{z}_1\in\GF(q)$. So,
\[ \Tr(ax^{\sqrt{q}})=a(\hat{x}_0^{\sqrt{q}}+\varepsilon^{\sqrt{q}} \hat{x}_1^{\sqrt{q}})+
	a(\hat{x}_0^{\sqrt{q}}-\varepsilon^{\sqrt{q}} \hat{x}_1^{\sqrt{q}})=
	2a\hat{x}_0^{\sqrt{q}}; \]
\[ x^{q+1}=
	(\hat{x}_0+\varepsilon \hat{x}_1)(\hat{x}_0-\varepsilon \hat{x}_1)=
  (\hat{x}_0^2-\varepsilon^2\hat{x}_1^2). \]
% \[ \Tr(mx)=(m_0+\varepsilon m_1)(x_0+\varepsilon x_1)+
% (m_0-\varepsilon m_1)(x_0-\varepsilon x_1)=
% 2(m_0x_0+\varepsilon^2 m_1x_1). \]
With this choice of $\varepsilon$, \eqref{eq21} becomes
\begin{equation}
	\label{eqodd1}
	a\hat{x}_0^{\sqrt{q}}-\hat{x}_0^2+\varepsilon^2\hat{x}_1^2+\hat{d}_0=0,
	%\varepsilon^2x_1^2+m_0x_0+\varepsilon^2m_1x_1 +
\end{equation}
to be solved for $(\hat{x}_0,\hat{x}_1)\in\GF(q)^2$.

Let $\Xi$ be the affine curve of
Equation~\eqref{eqodd1} over $\GF(q)$.
The number of  $\GF(q)$-rational points $P=(\hat{x}_0,\hat{x}_1)$ of $\Xi$ is the number $M_{a,d}$.
Rewrite~\eqref{eqodd1} as
\begin{equation}\label{eqodd2}
	aX^{\sqrt{q}}-X^2+\varepsilon^2 Y^2+t=0.
\end{equation}
Let now $\{1,\eta\}$ be a basis of $\GF(q)$ over $\GF(\sqrt{q})$. As before we can choose $\eta \in \GF(q)\setminus \GF(\sqrt{q})$ such that $\eta^{\sqrt{q}}=-\eta$ and $\eta^2$ is a primitive element in $\GF(\sqrt{q})$.
Set $X=X_0+\eta X_1$ and $\varepsilon^2=e$. So,
\[ X^2=(X_0+\eta X_1)^2=(X_0^2+2\eta X_0X_1+\eta^2 X_1^2); \]
\begin{multline*} \varepsilon^2 Y^2=(e_0+\eta e_1)(Y_0^2+2\eta Y_0Y_1+\eta^2 Y_1^2)=\\
	(e_0Y_0^2+e_0\eta^2Y_1^2+2\eta^2e_1Y_0Y_1)+\eta(e_1Y_0^2+\eta^2e_1Y_1^2+2e_0Y_0Y_1).
\end{multline*}
% \begin{multline*}
%  \varepsilon sY=(e_0+\eta e_1)(s_0+\eta s_1)(Y_0+\eta Y_1)=\\
%  (e_0s_0+\eta^2 e_1s_1)Y_0+\eta^2(e_1s_0+e_0s_1)Y_1+
%  \eta(e_0s_0+\eta^2 e_1s_1)Y_1+\eta(e_0s_1+e_1s_0)Y_0.
%\end{multline*}
%(We won't use this, but $a^{\sqrt{q}+1}=1$ implies $a_0^2-\eta^2a_1^2=1$.)
Thus, Equation~\eqref{eqodd2} is equivalent to the system of the following two equations
\begin{align}
	\label{eq11}
	t_0-X_0^2-\eta^2 X_1^2+e_0Y_0^2+e_0\eta^2 Y_1^2+2 \eta^2 e_1Y_0Y_1+a_0X_0-\eta^2a_1X_1=0, \\
	\label{eq12}
	t_1-2X_0X_1+e_1Y_0^2+\eta^2 e_1 Y_1^2+2e_0Y_0Y_1+a_1X_0-a_0X_1
	=0.
\end{align}
As both are  non-homogeneous quadratic equations in $(X_0,X_1,Y_0,Y_1)\in\GF(\sqrt{q})^4$,
the solutions of the system
correspond to the affine points of the intersection of two
quadratic hypersurfaces $\cQ_1$ and $\cQ_2$ of $\PG(4,\sqrt{q})$.
Thus, the number to determine is $M_{a,d}=|(\cQ_1\cap\cQ_2)\setminus\Sigma_{\infty}|$
where $\Sigma_{\infty}$ denotes the hyperplane at infinity.
To count the number of these points we  first  determine the number of  points
at infinity of $\cQ_1 \cap \cQ_2$.
The matrices associated to the quadrics $\cQ_1^{\infty}=\cQ_1 \cap \Sigma_{\infty}$ and $\cQ_2^{\infty}=\cQ_2 \cap \Sigma_{\infty}$ are
respectively
\begin{equation}\label{matr1}
  A_{1}^{\infty}=\begin{pmatrix}
		-2 & 0        & 0          & 0          \\
		0  & -2\eta^2 & 0          & 0          \\
		0  & 0        & 2e_0       & 2e_1\eta^2 \\
		0  & 0        & 2e_1\eta^2 & 2e_0\eta^2
	\end{pmatrix}
\end{equation}
and
\begin{equation}\label{matr2}
	A_{2}^{\infty}=\begin{pmatrix}
		0  & -2 & 0     & 0           \\
		-2 & 0  & 0     & 0           \\
		0  & 0  & 2e_1  & 2e_0        \\
		0  & 0  & 2 e_0 & 2e_1 \eta^2
	\end{pmatrix}.
\end{equation}
We have $\det(A_{1}^{\infty})=16\eta^4(e_0^2-\eta^2e_1^2)$ and $\det(A_{2}^{\infty})=16(e_0^2-e_1^2\eta^2)$.
Since $\eta \in \GF(q)\setminus \GF(\sqrt{q})$ it follows that  $\det(A_{1}^{\infty})\neq 0 \neq \det(A_{2}^{\infty})$ and hence $\cQ_1^{\infty}$ and $\cQ_2^{\infty}$ are always non-singular.
Now consider the pencil $\cF$ spanned by the two quadrics $\cQ_1^{\infty}$
and $\cQ_2^{\infty}$ over $\GF(q)$.
A generic quadric $\cQ_{\xi,\lambda}$ with $(\xi,\lambda)\in\GF(\sqrt{q})^2
	\setminus\{(0,0)\}$
of $\cF$ has matrix
\begin{equation}\label{matr3}
	A_{\xi,\lambda}^{\infty}=\begin{pmatrix}
		-2\xi     & -2\lambda   & 0                           & 0                                \\
		-2\lambda & -2\eta^2\xi & 0                           & 0                                \\
		0         & 0           & 2e_0\xi +2\lambda e_1       & 2\eta^2e_1\xi+2\lambda e_0       \\
		0         & 0           & 2\eta^2e_1\xi+ 2\lambda e_0 & 2e_0\eta^2\xi+2e_1 \eta^2\lambda
	\end{pmatrix}
\end{equation}
whose determinant is
\begin{equation}
	\label{deteq}
	\det(A_{\xi,\lambda}^{\infty})=
	16(\lambda^2-\eta^2\xi^2)^2(e_0^2-e_1^2\eta^2).
\end{equation}
Observe that $\det(A_{\xi,\lambda}^{\infty})=0$ if and only if either $\eta^2\xi^2=\lambda^2$ or
$\eta^2=(e_0/e_1)^2$.
Since $\eta\not\in\GF(\sqrt{q})$ but $\eta^2\in\GF(\sqrt{q})$, we have that
$\eta^2$ is a
non-square in $\GF(\sqrt{q})$; so, neither of these
conditions is possible.
Then $\det(A_{\xi,\lambda}^{\infty})\neq 0$, $\forall (\xi, \lambda) \in \GF(\sqrt{q})^2\setminus\{(0,0)\}$.
Since $(e_0^2-\eta^2 e_1^2)=\varepsilon^{2(p+1)}$ is a non-square in $\GF(\sqrt{q})$, all quadrics in $\cF$ are elliptic and hence with $q+1$ points.
By the argument above,
the generic quadric $\cQ_{\xi,\lambda}$ of the pencil $\overline{\cF}$ spanned by $\cQ_1$ and $\cQ_2$ has rank at least $4$.
Let now $\overline{C}$ be  the base
curve of the pencil $\overline{\cF}$. The locus $\overline{C}$ is a quartic curve and  the number of its affine rational points over $\GF(\sqrt{q})$ is $M_{a,d}$.
If we denote by $C$ the base locus of the pencil $\cF$  then we have $|\overline{C}|=M_{a,d}+|C|$.

Observe that
\[|\Sigma_{\infty}|= |C|+(\sqrt{q}+1)(q+1-|C|), \]
so we get
\[ |C|=0. \]
On the other hand, let $r_5$,  $r_4^-$  be respectively the numbers of non singular quadrics  and elliptic cones in $\overline{\cF}$. Then,
\[ |PG(4,\sqrt{q})|=|\overline{C}|+r_5[(q+1)(\sqrt{q}+1)-|\overline{C}|]+r_4^-[\sqrt{q}(q+1)+1-|\overline{C}|],\]
and
\[r_5 +r_4^-=\sqrt{q}+1.\]
So $|\overline{C}|=r_5\sqrt{q}+1$  namely
$M_{a,d}\equiv 1\pmod {\sqrt{q}}$.
%Since $r_4^-$ is at most $5$
We obtain that $M_{a,d}=(\sqrt{q}+1-r_4^-)\sqrt{q}+1$.
%with $\alpha\in \{0,1,2,3,4,5\}$.
We are now going to prove that $r_4^-\leq 3$.
To this aim, we
compute the number of homogeneous solutions in $\GF(\sqrt{q})^2\setminus \{(0,0)\}$ of the following equation
in $(\xi,\lambda)$:
\begin{small}
	\begin{equation}\label{matr4}
		\setlength\arraycolsep{1pt}
		\det\begin{pmatrix}
			-2\xi              & -2\lambda                   & 0                            & 0                                & \xi a_0+\lambda a_1             \\
			-2\lambda          & -2\eta^2\xi                 & 0                            & 0                                & -(a_1\eta^2\xi+\lambda a_0) \\
			0                  & 0                           & 2\xi e_0 +2\lambda e_1       & 2\eta^2e_1\xi+2\lambda e_0       & 0                           \\
			0                  & 0                           & 2\eta^2\xi e_1+ 2\lambda e_0 & 2e_0\xi\eta^2+2e_1 \eta^2\lambda & 0                           \\
			a_0\xi+\lambda a_1 & -(a_1\eta^2\xi+\lambda a_0) & 0                            & 0                                & 2(t_0\xi+\lambda t_1)
		\end{pmatrix}=0,
	\end{equation}
\end{small}
% \begin{equation}\label{quad2}
% Rearranging the rows and columns in~\eqref{matr4} we get
% \begin{small}
% 	\setlength\arraycolsep{1pt}
% 	\[
% 		\det\begin{pmatrix}
% 			-2\xi               & -2\lambda                    & \xi a_0+\lambda a_1          & 0                           & 0                                  \\
% 			-2\lambda           & -2\xi\eta^2                  & -(\xi a_1\eta^2+\lambda a_0) & 0                           & 0                                  \\
% 			\xi a_0+\lambda a_1 & -(\xi a_1\eta^2+\lambda a_0) & 2(\xi t_0+\lambda t_1)       & 0                           & 0                                  \\
% 			0                   & 0                            & 0                            & 2\xi e_0 +2\lambda e_1      & 2\xi\eta^2e_1+2\lambda e_0         \\
% 			0                   & 0                            & 0                            & 2\xi\eta^2e_1+ 2\lambda e_0 & 2\xi e_0\eta^2+2\lambda e_1 \eta^2
% 		\end{pmatrix}=0,
% 	\]
% \end{small}
that is
\begin{small}
	\[
		4(e_0^2-e_1^2\eta^2)(\lambda^2-\eta^2\xi^2)\det\begin{pmatrix}
			-2\xi               & -2\lambda                    & \xi a_0+\lambda a_1          \\
			-2\lambda           & -2\xi\eta^2                  & -(\xi a_1\eta^2+\lambda a_0) \\
			\xi a_0+\lambda a_1 & -(\xi a_1\eta^2+\lambda a_0) & 2(\xi t_0+\lambda t_1)
		\end{pmatrix}=0.\]
\end{small}
\!Since $(\lambda^2-\eta^2\xi^2)$ is irreducible over $\GF(\sqrt{q})$, the solutions of Equation~\eqref{matr4}
correspond to those of
\begin{equation}\label{eee} \det\begin{pmatrix}
		-2\xi               & -2\lambda                    & \xi a_0+\lambda a_1          \\
		-2\lambda           & -2\xi\eta^2                  & -(\xi a_1\eta^2+\lambda a_0) \\
		\xi a_0+\lambda a_1 & -(\xi a_1\eta^2+\lambda a_0) & 2(\xi t_0+\lambda t_1)
	\end{pmatrix}=0.
\end{equation}
Suppose now that \eqref{eee} does not admit solutions of the form $(0,\lambda)$. Then
it is possible to divide by $\xi$, obtaining an equation of degree $3$
in $\lambda$, which (clearly) has at most $3$ different solutions in $\GF(\sqrt{q})$; so $0\leq r_4^- \leq 3$.
On the contrary, if $(0,\lambda)$ is a solution of~\eqref{eee}.
Then, $-8t_1+4a_0a_1=0$, that is $a_0a_1=2t_1$.
Replacing this in~\eqref{eee}, we get that the
number of homogeneous solutions of~\eqref{eee} is $1$ plus the number of isotropic points of the $1$-dimensional quadric with matrix
\[ \begin{pmatrix}
		-4t_0+3a_1\eta^2+3a_0 & 4a_0a_1\eta^2                  \\
		4a_0a_1\eta^2         & 4\eta^2t_0+a_1\eta^4+a_0\eta^2
	\end{pmatrix}. \]
This quadric has at most $2$ points, so $1\leq r_4^-\leq 3$.
Consequently, we obtain
\[M_{a,d}\in \{q+1-2\sqrt{q},q+1-\sqrt{q},q+1,q+\sqrt{q}+1\}.\]

\paragraph{The even $q$ case.}
%Now  assume $p=2$.
%In order to complete the proof of Theorem~\ref{main-1} we need to determine the number $N$
%of solutions
%$(x,y)\in\GF(q^2)\times \GF(q^2)$ of
%the system
%\begin{equation}
%	\left\{\begin{array}{l}\label{sis2}
%		\Tr(y+ax^{\sqrt{q}})=\N(x)\\
%		y=mx+d
%	\end{array}\right.
%\end{equation}
%We argue in a similar way to the case before;  in this case Equation \ref{disp} becomes
%\begin{equation}
%	\label{eq02}
%	\Tr(ax^{\sqrt{q}}+m^{q+1}+am^{\sqrt{q}q}+d)=\N(x).
%\end{equation}
%to have to study the number of solutions in $\GF(q^2)$ of the
%following equation in $x$:
%\begin{equation}
%	\label{eq01}
%a^qx^{\sqrt{q}q}+m^qx^{q}+ax^{\sqrt{q}}+mx+d^q+d=x^{q+1},
% a^qx^{2q}+m^qx^q+b^qx^{q+1}+d^q=bx^{q+1}+ax^2+mx+d.
%\end{equation}
%or, in other words
%\[ \Tr(ax^{\sqrt{q}}+mx+d)=\N(x), \]
%If we rewrite Equation~\eqref{eq01} putting $x+m^q$ instead of $x$,
%we get
%the polynomial
%\[ \Tr(a(x+m^q)^{\sqrt{q}}+m(x+m^q)+d)=\N(x+m^q) \]
%which becomes
%\begin{equation}
%	\label{eq02}
%	\Tr(ax^{\sqrt{q}}+d)+m^{q+1}=\N(x).
%\end{equation}
%Since $m^{q+1}\in \GF(q)$, there exists  $b\in \GF(q^2)$ such that $b^q+b=m^{q+1}$. Setting $f=d+b$,
%Equation~\eqref{eq02} becomes
%\[ \Tr(ax^{\sqrt{q}}+f)=\N(x).\]
%So, as $f$ varies in
%$\GF(q^2)$,
%we can assume $m=0$  in \eqref{eq02}. Thus,
Arguing in the same way as for the $q$ odd case, we see that we
have to determine the number of solutions of the following equation
\begin{equation}
	\label{eq021}
	\N(a)^{1-\sqrt{q}}x^{\sqrt{q}q}+\N(a)^{1-\sqrt{q}} x^{\sqrt{q}}+d^q/N(a)^{\sqrt{q}}+d/\N(a)^{\sqrt{q}}-x^{q+1}=0,
\end{equation}
with $a\in \GF(q)^*$ and $d \in \GF(q^2)$.
%\begin{equation}
%	\label{eq021}
%	a^qx^{\sqrt{q}q}+ax^{\sqrt{q}}+d^q+d-x^{q+1}=0.
%	%a^qx^{2q}+m^qx^q+b^qx^{q+1}+d^q=bx^{q+1}+ax^2+mx+d.
%\end{equation}
%As for $q$ odd,  we can substitute $x$ with $x/w$ for
%a suitable $w\in\GF(q^2)$ such that
%$a/w^{\sqrt{q}}=t\in\GF(q)$. Clearly this does not change the number of solutions
%of~\eqref{eq021}. Set now $c:=tw^{q+1}(\in \GF(q))$;
%thus we end up with the following
%\[ cx^{\sqrt{q}q}+cx^{\sqrt{q}}+w^{q+1}d^q+w^{q+1}d-x^{q+1}=0.\]
%This implies that we can
%assume $a \in \GF(q)^*$ in~\eqref{eq021} and hence  in~\eqref{sis2}.

Fix a primitive element $\eta$ of $\GF(q^2)\setminus \GF(q)$
such that  $\eta^{q}+\eta=1$ and $\eta^2+\eta+\nu=0$ where
$\nu \in \GF(q)\setminus \{1\}$ and $ \Tr_{\GF(2)}(\nu)=1$.
Then,
$\{1,\eta\}$ is a basis of
$\GF(q^2)$ regarded as a vector space over $\GF(q)$.
The elements of $\GF(q^2)$ shall be written as linear
combinations with respect to this basis, that is,
$z=\hat{z}_0+\hat{z}_1\eta$, with $z\in \GF(q^2)$ and
$\hat{z}_0,\hat{z}_1\in\GF(q)$.

With our choice of $\eta$,
Equation~\eqref{eq021} becomes
\begin{equation}
	\label{eq0odd1}
	a\hat{x}_1^{\sqrt{q}}+\hat{d_1}=\hat{x}_0^2+\hat{x}_0\hat{x}_1+\nu \hat{x}_1^2,
	%\varepsilon^2x_1^2+m_0x_0+\varepsilon^2m_1x_1 +
\end{equation}
to be solved for $(\hat{x}_0,\hat{x}_1)\in\GF(q)^2$.
Rewrite \eqref{eq0odd1} as
\begin{equation}\label{eq0odd3}
	aY^{\sqrt{q}}+t=X^2+XY+\nu Y^2,
\end{equation}
and
let $\{1,\gamma \}$ be a basis of $\GF(q)$ over $\GF(\sqrt{q})$. We
can choose $\gamma \in \GF(q)\setminus \GF(\sqrt{q})$ such that $\gamma^{\sqrt{q}}+\gamma=1$ and $\gamma^2+\gamma+\nu'=0$  where $\nu'$ is an element in $\GF(\sqrt{q})\setminus \{1\}$ whose absolute trace is $1$.
Setting $X=X_0+\gamma X_1$, $Y=Y_0+\gamma Y_1$ and $\nu=\nu_0+\gamma \nu_1$, \eqref{eq0odd3}
is equivalent to the system of the following two equations:
\begin{align}
	\label{eq011}
	X_0^2+\nu'^2 X_1^2+\nu_0 Y_0^2+\nu'(\nu_0+\nu_1)Y_1^2+ & X_0Y_0+\nu' X_1Y_1+             \\
  \nonumber                                              & +a_0Y_0+(a_0+\nu'a_1)Y_1+t_0=0, \\
  \label{eq012}
	X_1^2+[(\nu_0+\nu_1)+\nu' \nu_1]Y_1^2+\nu_1Y_0^2+      & X_0Y_1+Y_0X_1+X_1Y_1+           \\
	\nonumber                                              & +a_0Y_1+a_1Y_0+t_1
	=0.\end{align}
As these are  non-homogeneous quadratic equations in $(X_0,X_1,Y_0,Y_1)\in\GF(\sqrt{q})^4$,
their solutions
correspond to the affine points of the intersection of two
quadratic hypersurfaces $\cC_1$ and $\cC_2$ of $\PG(4,\sqrt{q})$.

We refer the reader to~\cite{EKM} for the theory of quadrics in
arbitrary (including $2$) characteristic.
Our approach to study~\eqref{eq011} and~\eqref{eq012} here is according
to~\cite[Section 1.2]{James}. In particular, we describe the quadrics
by means of matrices where we replaced each of the terms $a_{ij}$'s by formal indeterminates $Z_{ij}$, then evaluate the discriminant and the Arf invariant as rational functions over the ring of integers ${\mathbb Z}$ and then finally we specialize the indeterminates $Z_{ij}$ to
$a_{ij}$ once more.
So, the matrices associated to the quadrics
$\cC_1^{\infty}:=\cC_1\cap \Sigma_{\infty}$ and $\cC_2^{\infty}:=\cC_2 \cap \Sigma_{\infty}$ are
\begin{equation}\label{matr01}
	P_{1}^{\infty}=\begin{pmatrix}
		2 & 0     & 1      & 0                  \\
		0 & 2\nu' & 0      & \nu'               \\
		1 & 0     & 2\nu_0 & 0                  \\
		0 & \nu'  & 0      & 2\nu'(\nu_0+\nu_1)
	\end{pmatrix}
\end{equation}
and
\begin{equation}\label{matr02}
	P_{2}^{\infty}=\begin{pmatrix}
		0 & 0 & 0      & 1                           \\
		0 & 2 & 1      & 1                           \\
		0 & 1 & 2\nu_1 & 0                           \\
		1 & 1 & 0      & 2[(\nu_0+\nu_1)+\nu' \nu_1]
	\end{pmatrix}.
\end{equation}
We have $\det(P_{1}^{\infty})=\nu'^2 \neq 0$ and $\det(P_{2}^{\infty})=1-4\nu_1$,
hence $\cC_1^{\infty}$ and $\cC_2^{\infty}$ are both non-singular.

Let \[B:=\begin{pmatrix}
		0  & 0  & 0 & 1 \\
		0  & 0  & 1 & 1 \\
		0  & -1 & 0 & 0 \\
		-1 & -1 & 0 & 0
	\end{pmatrix}. \]
Since $\displaystyle\frac{|B|-1+4\nu_1}{4|B|}=\nu_1$ and
\begin{equation}
	\label{tre}
	\Tr_{\GF(\sqrt{q})\mid \GF(2)} (\nu_1)= \Tr_{\GF(\sqrt{q})\mid \GF(2)} (\nu+\nu^{\sqrt{q}})=\Tr_{\GF(q)\mid \GF(2)} (\nu)=1,
\end{equation}
we have that $\cC_{2}^{\infty}$ is an elliptic quadric.

Now consider the pencil $\cF$ generated by the quadrics $\cC_1^{\infty}$
and $\cC_2^{\infty}$ over $\GF(q)$.
A generic quadric $\cC_{\xi,\lambda}\in\cF$ has matrix
\begin{equation}\label{matr3}
	A_{\xi,\lambda}=\begin{pmatrix}
		2\xi    & 0                  & \xi                        & \lambda                                            \\
		0       & 2(\lambda+\xi\nu') & \lambda                    & \lambda+\xi\nu'                                    \\
		\xi     & \lambda            & 2(\xi\nu_0 +\lambda \nu_1) & 0                                                  \\
		\lambda & \lambda+\xi\nu'    & 0                          & 2[(\lambda+\xi\nu')(\nu_0+\nu_1)+\lambda\nu'\nu_1]
	\end{pmatrix}
\end{equation}
whose determinant is
\begin{equation}
	\det(A_{\xi,\lambda})=(\xi^2\nu'+\xi\lambda+\lambda^2)^2.
\end{equation}
Since $\Tr_{\GF(\sqrt{q})\mid \GF(2)}\nu'=1$, we have
that
$\forall (\xi,\lambda)\in GF(q)^2\setminus\{(0,0)\}, \det(A_{\xi,\lambda})\neq 0$.
Thus, each  $\cC_{\xi,\lambda}\in \cF $ is non-singular.
For $\xi=0$ we know that $\cC_{0,\lambda}$ is elliptic. We want to prove that $\cC_{\xi,\lambda}$ is elliptic also for $\xi\neq 0$.
Thus we  can assume $\xi=1$ and set
\[B_{\lambda}:= \begin{pmatrix}
		0        & 0               & 1       & \lambda      \\
		0        & 0               & \lambda & \lambda+\nu' \\
		-1       & -\lambda        & 0       & 0            \\
		-\lambda & -(\lambda+\nu') & 0       & 0
	\end{pmatrix}
	.\]
We get  that
\[\frac{|B_{\lambda}|-|C_{\lambda}|}{4|B_{\lambda}|}=\nu_1\]
and hence, from~\eqref{tre},   that $\cC_{1,\lambda}$ is an elliptic quadric for all $\lambda \in \GF(\sqrt{q})$.
This means that any element of $\cF$ is an elliptic quadric.
A straightforward computation shows  that the determinant of the matrix
\begin{small}
	\begin{equation}\label{matr3}
		\setlength\arraycolsep{1pt}
		\overline{A}_{\xi,\lambda}=\begin{pmatrix}
			2\xi    & 0                  & \xi                        & \lambda                                            & 0                           \\
			0       & 2(\lambda+\xi\nu') & \lambda                    & \lambda+\xi\nu'                                    & 0                           \\
			\xi     & \lambda            & 2(\xi\nu_0 +\lambda \nu_1) & 0                                                  & a_0\xi+\lambda a_1          \\
			\lambda & \lambda+\xi\nu'    & 0                          & 2[(\lambda+\xi\nu')(\nu_0+\nu_1)+\lambda\nu'\nu_1] & a_0(\lambda+\xi)+\xi\nu'a_1 \\
			0       & 0                  & \xi a_0+\lambda a_1        & a_0(\lambda+\xi)+\xi\nu'a_1                        & 2(\xi t_0+\lambda t_1)
		\end{pmatrix}
	\end{equation}
\end{small}
\!\!is the product of $\lambda^2+\lambda\xi+\nu'\xi^2$, which is an irreducible polynomial  over $\GF(\sqrt{q})$ by a homogeneous polynomial $p(\xi,\lambda)$ of degree $3$. If
\begin{equation}\label{sing}
	\det(\overline{A}_{\xi,\lambda})=0
\end{equation}does not admit solutions of the form $(0,\lambda)$  then we can put $\xi=1$ in Equation~\eqref{sing} and we end up with an equation of degree $3$ in $\lambda$ which has at most $3$ solutions. This implies that there are at most three singular quadrics in the pencil
$\overline{\cF}$ generated by $\cC_1$ and $\cC_2$.
If $(0,\lambda)$ is a solution of Equation~\eqref{sing}
%$det(\overline{A}_{(\xi,\lambda)})=0$
then $\det(\overline{A}_{\xi,\lambda})=\xi(\xi\lambda^2+\lambda\xi+\nu'\xi^2)h(\xi,\lambda)$ where $h(\xi,\lambda)$ is a homogeneous polynomial of degree $2$.
As before if $h(\xi,\lambda)$ is not divisible by $\xi$, we obtain with
a similar argument that $h(1,\lambda)$ has at most $2$ roots and we see again that in total there are at most $3$ singular quadrics in $\overline{\cF}$, one of them being $\cC_2$. So we have the following intersection numbers with non-vertical affine lines also for $q$ even:
\[q+1-2\sqrt{q},q+1-\sqrt{q},q+1,q+\sqrt{q}+1.\]
This completes the proof. \qed

%\section{A family of point sets in  $\PG(2,q^2)$ with few intersection numbers}
\label{nonst}

\subsection{Proof of Theorem~\ref{main-1}}
\label{pc32}
% Here we want to describe a family of regular sets of pointed type with very few affine types by using
% Theorem~\ref{t:32}.
%In order to simplify the computations, we can assume the Hermitian curve
%$\cH$ to have affine equation
%\begin{equation}
%	\label{nuovo6} y^q+y=2x^{q+1}.
%\end{equation}
%for $q$ odd. For $q$ even we take $\cH$ of equation
The Hermitian curve
\begin{equation}
	\label{nuovo7} \cH \colon y^q+y=x^{q+1}
\end{equation}
and the curve
%We shall denote throughout the paper by $\cH$ both the aforementioned
%affine curve and its projective closure (obtained by adding the point
%$Y_{\infty}=[(0,1,0)]$) as we believe that this should not engender
%any confusion to the reader; however in the text of the theorems
%we shall specify whether the result holds in a projective on an
%affine setting.
% It is well known that Hermitian curves are $2$-intersection sets with
% respect to projective lines, the intersection numbers being $1$ and $q+1$.
% In particular any affine
% line of the form $y=ax+b$ meets $\cH$ in either $1$ or
% $q+1$ points, while the
% ``vertical'' lines of equation $x=\alpha$ meet $\cH$ in
% $q$ affine points.
%
%We study the intersections of $\cH$ with curves
%$\cC_a(m,d)$ of degree $\sqrt{q}$ whose equation is of the form
\[ \cC(a,m,d) \colon y=ax^{\sqrt{q}}+mx+d, \]
with $a,m,d\in\GF(q^2)$ and $a\neq 0$,
share a unique point on the line at infinity, namely the point $(\infty)$.
On the other hand, the number of affine points lying on both curves
is the the same as the number of solutions of the system of equations
\begin{equation}
	\left\{\begin{array}{l}
		\Tr(y+ax^{\sqrt{q}})=\N(x), \\
		y=mx+d.
	\end{array}\right.
\end{equation}
This is the same as the number of common points of the line $y=mx+d$ and the curve $\Gamma_a$. Then the result follows from Theorem \ref{t:32}.

We propose a general conjecture.
\begin{conjecture}
	Let $p$ be a prime, $h\geq2$ and $q=p^{2h}$. Then the
	affine Hermitian curve $\cH(q^2)$ of $\AG(2,q^2)$ meets the curves
	$\cX(a,m,d) \colon y=ax^p+mx+d$ in
	%$N\equiv1\pmod p$
	$1$ modulus $p$ affine points.
\end{conjecture}

\begin{remark}
	\label{final}
	As we saw in the proof of Theorem \ref{t:32}, the number of lines with slope $m\neq \infty$ and meeting $\Gamma_a$ in $k_{\alpha}:=(\sqrt{q}+1-\alpha)\sqrt{q}+1$, $\alpha \in \{0,1,2,3\}$ points depends on the parameter $a$. In fact, for the same $q$, with different choices of $a$, we may get point sets with different number of $k_{\alpha}$-secants and hence nonequivalent constructions.

	The number of $k_0$, $k_1$, $k_2$, $k_3$-secants of $\Gamma_a$ with slope $m\neq \infty$ is, respectively is
	\begin{itemize}
        \item either $0$, $2^2 \cdot 3$, $0$ , $2^2$, or  $2^2$, $0$, $2^2\cdot 3$, $0$ when $q=2^2$ (cf. Remark \ref{dafare}),
        \item $3^2\cdot 2$, $3^2 \cdot 3$, $3^2 \cdot 3$, $3^2$ when $q=3^2$,
        \item $4^2\cdot 4$, $4^2\cdot 6$, $4^2 \cdot 4$ ,$4^2 \cdot 2$ when $q=4^2$,
        \item either $5^2\cdot 6$, $5^2\cdot 12$, $5^2 \cdot 3$, $5^2 \cdot 4$, or
          $5^2\cdot 7$, $5^2 \cdot 9$, $5^2 \cdot 6$, $5^2 \cdot 3$, when $q=5^2$.
	\end{itemize}
	There are two combinatorially different examples also for $q=11^2$ and $q=17^2$.
\end{remark}

\section{A class of  $\sqrt{q}$-divisible codes over $\GF(q^2)$ and codes with a simple weight enumerator modulus a $q$-power}
\label{codes}
In this section
 apply the usual construction of codes arising from projective
systems to the curve $\Gamma_a$.
More in detail, we
construct a $3\times (q^3+1)$ generator matrix $G$ for a
code by taking as columns
the  coordinates of the points of the algebraic curve $\Gamma_a$ with
Equation~\eqref{sss}. The order in which the points are taken is not
relevant, as all codes thus obtained are equivalent.
%for $q$ odd and with equation \eqref{sssp} for $q$ even, where $a\in GF(q^2)^*$.

The code $\cC(\Gamma_a)$ having $G$ as generator matrix is called \emph{the  projective code generated from} $\Gamma_a$.
The spectrum of the intersections of $\Gamma_a$  with the lines of $\PG(2, q^2)$ is related to the list of the weights  $w_i$ of
the associated code; furthermore the minimum Hamming weight
of $\cC(\Gamma_a)$ is
\[w(\Gamma_a)=|\Gamma_a|-\max\{|\Gamma_a \cap \ell|: \ell \ \text{is a line of}\ \PG(2,q^2)\}.\]

Since $|\Gamma_a|=q^3+1$  it is now easy to see that $\cC(\Gamma_a)$ is a $[q^3+1,3,q^3-q-\sqrt{q}]_{q^2}$-linear code.
Also, $\cC(\Gamma_a)$ has  just $5$ weights, that is:
\[w_1=q^3-q-\sqrt{q}, w_2=q^3-q, w_3=q^3-q+\sqrt{q}, w_4=q^3-q+2\sqrt{q}, w_5=q^3\]
which are all divisible by $\sqrt{q}$.
Furthermore, for $q=4$, $w_4=w_5$ and the corresponding $\cC(\Gamma_a)$ is  either a $[65,3,60]_{16}$-linear code with two non-zero weights   or a  $[65,3,58]_{16}$-linear code with just $4$ non-zero weights (cf. Remark 3.1).

We  define the \emph{intersection enumerator} of the
projective curve arising from $\Gamma_a$ as the
polynomial
\[ \iota(x):=\sum_{\ell\text{ line of }PG(2,q^2)} x^{|\ell\cap\Gamma_a|}=\sum_i e_ix^i. \]
% Clearly, $\iota(1)=q^3+1$.
Denote by $A_i$ the number of codewords
of $\cC(\Gamma_a)$ with Hamming weight $i$. The (Hamming) weight enumerator is defined as the polynomial
\[1 + A_1x +\dots + A_m x^m;\]
this polynomial gives a great deal of information about the code and is an important invariant. Also, it is used in order to estimate the probability
of a successful decoding when there are more than $2d+1$ errors, $d$ being
the minimum distance of the code.

If $\iota(x)$ is the intersection enumerator of $\Gamma_a$,
then the weight enumerator of $\cC(\Gamma_a)$ is
\begin{equation}
	\label{weq} w(x)=1+(q^2-1)\sum e_i x^{q^3+1-i}.
\end{equation}

From Theorem~\ref{t:32}, it follows immediately that the only non-zero coefficients
in $e_i$ are those for $i\in\{1,q-2\sqrt{q}+1,q-\sqrt{q}+1,q+1,q+\sqrt{q}+1\}$.
Also, the only line meeting $\Gamma_a$ in exactly one point is the line
at infinity, and
the $q^2$ vertical lines of $\AG(2,q^2)$ meet
$\Gamma_q$ in $q+1$ points; so $e_1=1$. It is easy to see that the following general result holds.

\begin{prop}
	Let $X$ denote point set of $\PG(2,q^h)$, let $\cC(X)$ denote the associated projective code and let $e_i$ denote the coefficient of $x^i$ in $\iota(x)$.
	\begin{enumerate}
		\item If $X$ is regular of pointed type then all coefficients $e_i$ with $i\neq 1$ are divisible by $q^h$, $e_1 \equiv 1 \pmod{q^h}$.
		\item If $h=2$ and $X$ is obtained from Theorem \ref{trnorm} (such as $\Gamma_a$) then all coefficients $e_i$ with $i\notin \{1,q+1\}$ are divisible by $q^3$,  $e_{q+1} \equiv q^2 \pmod{q^3}$ and $e_{1} \equiv 1 \pmod{q^3}$.
		\item If $X=S'\cup\{(\infty)\}$ is obtained from a regular set $S$ of pointed type $[t;m_1,\ldots,m_g]$ of $\PG(2,q)$ as in Theorem \ref{largerplane}, then all coefficients $e_i$ with $i\notin \{1,q^st+1\}$ are divisible by $q^{2s+1}$, $e_{q^{s}t+1}=q$, $e_1 \equiv q^h-q+1 \pmod{q^{2s+1}}$.
	\end{enumerate}	\qed
\end{prop}

The result above has the following immediate consequence.

\begin{prop}
	If $X$ is regular of pointed type in $\PG(2,q^h)$, $\cC(X)$ denote the associated projective code and $e_i$ denote the coefficient of $x^i$ in $\iota(x)$, then the weight enumerator $1+(q^h-1)\sum e_i x^{|X|-i}$ of $\cC(X)$ modulus $q^h$ equals \[1-x^{|X|-1}.\]
	If $h=2$ and $X$ is obtained from Theorem \ref{trnorm} (such as $\Gamma_a$) then the weight enumerator of $\cC(x)$ modulus $q^3$ equals
	\[1-q^2 x^{q^3-q}+(q^2-1)x^{q^3}.\]
	If $X=S'\cup \{(\infty)\}$ is obtained from a regular set $S$ of pointed type $[t;m_1,\ldots,m_g]$ of $\PG(2,q)$ as in Theorem \ref{largerplane}, then the weight enumerator of $\cC(x)$ modulus $q^{2s+1}$ equals
	\[1+(q^{h+1}-q)x^{tq^{s}(q-1)}+(-q^{h+1}+q-1)x^{tq^{s+1}}.\]
	\qed
\end{prop}

%\begin{prop}
%	Let $a_1,a_2,\dots,a_5$ be the number of words of $\cC(\Gamma_a)$ of weight
%	$w_i$. Then $a_5=(q^2-1)$ while $a_i$ for $i=1,\dots,4$ is divisible
%	by $q^2$. In particular, the weight enumerator of $\cC(\Gamma_a)$
%	modulus $q^2$ is $1-x^{q^3}$.
%\end{prop}
%\begin{proof}
%	Let $\iota(x)=\sum_{i}e_ix^i$ be the intersection enumerator
%	of $\Gamma_a$. Then the weight enumerator of $\cC(\Gamma)$ is
%        given by~\eqref{weq}. Since, by
%        Proposition~\ref{cei}, $q^2$ divides each of the $e_i$
%	barring $e_1$ which is exactly $1$, we get
%	\[ w(x)\pmod{q^2}=1-x^{q^3}. \]
%\end{proof}
%\begin{corollary}
%  The code $\cC(\Gamma_a)$ is quasi-minimal.
%\end{corollary}
%\begin{proof}
% Recall, see~\cite{CMP}, that a code is quasi-minimal if and only if
% each of its codewords is uniquely determined (up to scalar multiple)
% by its support. In the setting of projective codes this is the
% same as to say that each hyperplane section $\Gamma_a\cap\Sigma$ uniquely
% determines the hyperplane $\Sigma$.
% As $k=3$, the hyperplanes are lines and if a line meet $\Gamma_a$ in
% at least $2$ distinct points, then it is uniquely determined.

%By the above argument on the intersection enumerator that $e_1=1$.
%So there is exactly one line meeting $\Gamma_a$ in one point and
%this line is uniquely determined. The thesis follows.
%\end{proof}

Observe that the codes $\cC(\Gamma_a)$
not only  have good parameters, but they turn also out to be  $\sqrt{q}$-divisible, see \cite{Kurz, WH}.
Incidentally, as the codes we consider are projective, their duals
are $[q^3+1,q^3-3,3]$-linear almost MDS codes (however they are not
NMDS), see \cite{DB}.

\section*{Acknowledgement}

This work was supported by the Italian National Group for Algebraic and Geometric Structures and their Applications (GNSAGA--INdAM)
and partially supported by the European Union under the Italian National Recovery and Resilience Plan (NRRP) of NextGenerationEU, partnership on “Telecommunications of the Future” (PE00000001 - program ‘‘RESTART’’, CUP: D93C22000910001).
The second author acknowledges the partial support of the National Research, Development and Innovation Office – NKFIH, grant no. K 124950.

\end{document}